\documentclass[11pt]{article}
\usepackage{amsmath}
\usepackage{amsthm}
\usepackage{amsfonts} 

\usepackage{mathabx}
\usepackage{amssymb}
\usepackage[pdftex]{graphicx} 
\usepackage[english]{babel} 
\usepackage[colorlinks = true,
            linkcolor = blue,
            urlcolor  = blue,
            citecolor = blue,
            anchorcolor = blue]{hyperref} 
\usepackage{calc} 
\usepackage{dsfont}
\usepackage{enumitem}
\usepackage[T1]{fontenc}

\usepackage{authblk}
\usepackage{physics}
\usepackage{tikz}
\usetikzlibrary{knots}
\allowdisplaybreaks
\usepackage[a4paper]{geometry} 
\usepackage[all]{nowidow}
\usepackage{lipsum} 
\hypersetup{ 	
	pdfsubject = {},
	pdftitle = {},
	pdfauthor = {}
}

\theoremstyle{plain}
\newtheorem{theorem}{Theorem}[section]

\newtheorem{lemma}[theorem]{Lemma}

\theoremstyle{definition}
\newtheorem{definition}[theorem]{Definition}

\theoremstyle{remark}

\newtheoremstyle{named}{}{}{\itshape}{}{\bfseries}{.}{.5em}{\thmnote{#3}#1}
\theoremstyle{named}

\title{Chainmail links and non-left-orderability}
\author{Zipei Nie \thanks{Lagrange Mathematics and Computing Research Center, Huawei. Email: niezipei@huawei.com.}}
	
\date{\today}

\begin{document} 
\maketitle
\begin{abstract}
   We prove that the alternating surgeries on flat fully augmented chainmail links yield total L-spaces. We also study the non-left-orderability of surgeries on the connected sum with an L-space knot using order detection.
\end{abstract}
\section{Introduction}
\subsection{Backgrounds on L-space conjecture}
An L-space is a rational homology $3$-sphere $Y$ with equality $\rank \widehat{{HF}}(Y) = |H_1(Y)|$. Here, $\widehat{{HF}}(Y)$ represents the hat version of Heegaard Floer homology, and $H_1(Y)$ is the first homology group of $Y$ with coefficients in $\mathbb{Z}$. A group $G$ is called left-orderable if it is nontrivial and admits a total order $\leq$ that remains invariant under left multiplication. In other words, for all $f,g,h \in G$, if $g \leq h$, then $fg \leq fh$. An L-space is called a total L-space if its fundamental group is not left-orderable. One aspect of the L-space conjecture \cite{boyer2013spaces} predicts that every L-space is a total L-space.

Several strategies addressing the conjecture exist in the literature. First, since the definition of L-space is connected to the Heegaard splitting, we can directly study the fundamental group derived from the Heegaard splitting. Pursuing this approach, Levine and Lewallen \cite{levine2013strong} introduced a specific type of Heegaard diagrams known as strong Heegaard diagrams, which yield total L-spaces. It is known \cite{greene2013spanning} that the double branched cover of a non-split alternating link is a strong L-space. Partial evidence \cite{greene2016strong} supports the possibility of establishing the converse statement.

Alternatively, since L-spaces cannot admit a coorientable taut foliation, one may seek to establish a taut foliation based on a left order within the fundamental group. Note that this direction also addresses an additional conjectural implication in the L-space conjecture. Recently, Li \cite{li2022taut} achieved this for irreducible $3$-manifolds with Heegaard genus two.

One may also speculate that an unknown topological characterization \cite[Question 11]{ozsvath2004heegaard} for L-spaces is necessary to provide crucial information for demonstrating the non-left-orderability. Such an extra topological property may also shed light on the Heegaard Floer Poincaré conjecture, which asserts that \cite[p. 39]{ozsvath2004lectures} an irreducible integer homology sphere is an L-space if and only if it is homeomorphic to $S^3$ or the Poincaré homology sphere. We illustrate this approach in Figure \ref{fig:framework}.

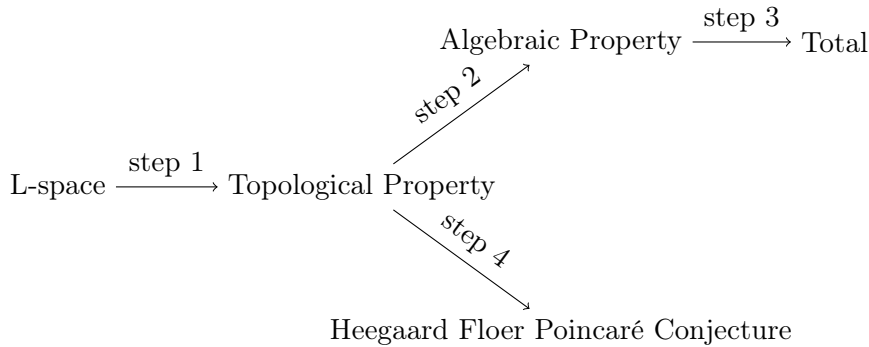
\begin{figure}
    \centering
    \begin{tikzpicture}[node distance=3cm, auto]
        \node (Lspace) {L-space};
        \node (TopProperty) [right of=Lspace , xshift=1cm] {Topological Property};
        \node (AlgProperty) [above right of=TopProperty,xshift=0.5cm,yshift=-0.2cm] {Algebraic Property};
        \node (NLO) [right of=AlgProperty,xshift=0.6cm] {Total};
        \node (HFPC) [below right of=TopProperty,xshift=0.5cm,yshift=0.2cm] {Heegaard Floer Poincaré Conjecture};
        
        \draw[->] (Lspace) -- node {step 1} (TopProperty);
        \draw[->] (TopProperty) --node[sloped] {step 2} (AlgProperty);
        \draw[->] (TopProperty) --node[sloped] {step 4} (HFPC);
        \draw[->] (AlgProperty) --node {step 3} (NLO);
    \end{tikzpicture}
    \caption{The Third Approach Framework}
    \label{fig:framework}
\end{figure}
For instance, the L-space conjecture was verified for all graph manifolds \cite{boyer2017foliations, hanselman2020spaces}; see also \cite{rasmussen2017space}. The topological structure of L-space graph manifolds was implicitly used in Boyer and Clay's work \cite{boyer2017foliations}, and it was explicitly expressed in \cite[Theorem 1.6]{rasmussen2017space}. Both \cite[Theorem 7]{hanselman2020spaces} and \cite[Theorem 1.6]{rasmussen2017space} offered the necessary gluing result for detecting L-spaces in graph manifolds. Hanselman, Rasmussen and Watson \cite[Theorem 1.14]{hanselman2016bordered} developed the theory of immersed curves to establish a general gluing result for arbitrary splicing along an incompressible torus boundary. This theory also provided a new proof of Eftekhary’s theorem \cite{eftekhary2018bordered} (see also \cite{hedden2016splicing}) that L-space homology spheres are atoroidal. In particular, every counterexample of the Heegaard Floer Poincaré conjecture is hyperbolic.

As to the non-left-orderability for L-space graph manifolds, Boyer and Clay considered two extra algebraic properties of a left order on the fundamental groups of $3$-manifolds with multiple torus boundary components. First, they introduced a notion of order detection (see \cite[Definition 4.2]{boyer2017foliations}), which could be viewed as a mixture of two boundary restrictions, namely (regular) detection and strong detection. The former one corresponds to the splicing operation along a torus boundary component, while the latter one corresponds to the Dehn filling. Second, they considered whether the fiber class is cofinal with respect to a given left order. If an element is cofinal with respect to a left order, then the conjugates of this element have the same sign (see \cite[Lemma 4.6]{boyer2017foliations}).

Establishing non-left-orderability via cofinal elements is highly effective for surgeries on L-space knots (or links with a small number of components), because the cofinality of a boundary element may be determined without specifying the surgery slope (see \cite[Conjecture 1.8]{boyer2022order}). This approach is also called fixed point method because it involves tracking the fixed points in the dynamical realization. Examples of this approach include proving non-left-orderability of the fundamental groups of $5$-fold cyclic branched covers of genus one two-bridge knots \cite{ba2019spaces} and L-spaces obtained from surgeries on $(1,1)$-knots \cite{nie2021explicit}. 

When considering L-spaces obtained from surgeries on L-space links with an unbounded number of components, the contribution of cofinal elements becomes marginal. Discovering more algebraic properties becomes necessary to prove non-left-orderability of their fundamental groups. Examples in this category include the double branched cover of non-split alternating links \cite{boyer2013spaces, greene2018alternating, ito2013non}, the $n$-th cyclic branched cover of certain two-bridge knots \cite{dkabkowski2005non}, and the $n$-th cyclic branched cover of the pretzel knot $P(3,-3,-2k-1)$ \cite{lin2022non}. 
\subsection{Main results}
Polyak \cite{polyak20143} introduced the terminology of chainmail links. It was shown \cite{matveev1994geometrical, agol2023chainmail} that every $3$-manifold could be obtained from a surgery on a flat fully augmented chainmail link, and that flat fully augmented chainmail links are L-space links. Igor conjectured on \cite[page 16]{agol2023chainmail} that the L-space surgeries described in \cite[Theorem 4.2]{agol2023chainmail} on alternating and augmented chainmail links have non-left-orderable fundamental groups. This paper aims to provide a proof for this conjecture.

\begin{theorem}\label{main-1}
    The fundamental group of a $3$-manifold obtained from alternating surgeries on flat fully augmented chainmail links is not left-orderable. 
\end{theorem}

The $3$-manifolds discussed in Theorem~\ref{main-1} are more general than those presented in \cite[Theorem 4.2]{agol2023chainmail} because we allow rational surgery slopes. Nevertheless, by making minor adjustments to the proof of \cite[Theorem 4.2]{agol2023chainmail}, one can demonstrate that the $3$-manifolds described in Theorem~\ref{main-1} are L-spaces.

When all the vertex weights of the chainmail graph are zero, the resulting $3$-manifold is homeomorphic to the connected sum of the double branched cover of an alternating link and $S^2\times S^1$, as shown in \cite[Subsection 2.2]{agol2023chainmail}, and the non-left-orderability follows from \cite{boyer2013spaces, greene2018alternating, ito2013non}. In cases where all but one vertex weight is zero, Agol (see \cite[Theorem 5.1]{agol2023chainmail}) proved the non-left-orderability by making minor modifications to Greene's proof \cite{greene2018alternating} regarding the non-left-orderability of fundamental groups of double branched covers of alternating links. Our proof of Theorem~\ref{main-1} also shares some similarities with Greene's proof.

Another technique we use was previously demonstrated in \cite{lin2022non}. The proof of \cite[Theorem 1.2]{lin2022non} relied on a localized version of the cofinal elements method. Specifically, \cite[Lemma 3.1]{lin2022non} established that, for any left total preorder $\le$, if $\{b_i\}$ is cofinal in $\{b_i\}\cup\{a_{i+1}^m:m\in\mathbb{Z}\}$, then $\{a_{i+1}^m b_i: m\in \mathbb{Z}\}$ is cofinal in $\{a_{i+1}^m b_i: m\in \mathbb{Z}\}\cup \{a_{i}^m: m\in \mathbb{Z}\}$. Furthermore, \cite[Lemma 3.2]{lin2022non} established that, if $\{b_i\}$ is cofinal in $\{b_i\}\cup\{a_{i+1}^m:m\in\mathbb{Z}\}$, then $\{b_{i-1}\}$ is cofinal in $\{b_{i-1}\}\cup\{a_{i}^m:m\in\mathbb{Z}\}$. This pattern was followed throughout the rest of the arguments. Our proof of Theorem~\ref{main-1} also relies on such propositions to complete the inductive step. The quantifiers in these propositions are essential. Without them, the complexity of the reasoning would grow exponentially.

In the remark after the proof of \cite[Theorem 3.2]{agol2023chainmail}, it was noted that the combination of any component from a negative alternating chainmail link with an L-space knot results in another L-space link. The L-spaces derived from surgeries on the newly formed L-space link can be obtained by splicing two Floer simple manifolds together. One can demonstrate the non-left-orderability of the fundamental groups of such L-spaces by \cite[Theorem 1.6]{boyer2022order}, \cite[Theorem 1.7]{boyer2022order} and order detection results for these Floer simple manifolds.

In general, the machinery of order detection allows us to generate much more total L-spaces through splicing and Dehn filling. We present gluing theorems for these topological operations, serving as the non-orderable counterpart to \cite[Theorem 7.6]{boyer2021slope} and \cite[Theorem 7.10]{boyer2022order}.

The paper is organized as follows.

In Section~\ref{sec:2}, we introduce a slight modification to the definition of order detection and provide gluing theorems.

In Section~\ref{sec:3}, we provide the definition of the flat fully augmented chainmail link and a presentation of its link group.

In Section~\ref{sec:4}, we provide the proof of an order detection result for the link complement of a flat fully augmented chainmail link.

\section*{Acknowledgements} The author would like to thank Ian Agol and Qiuyu Ren for helpful discussions.

\section{Order detection}\label{sec:2}

We introduce a few modifications to the concept of order detection introduced in \cite{boyer2017foliations} and \cite{boyer2022order}. First, we replace left orders with (proper) left total preorders. A left total preorder on a group $G$ is a proper total order on $G$ that remains invariant under left multiplication. The primary advantage of abandoning the antisymmetry axiom is that it allows us to separate the L-space detection problem from determining the irreducibility of the $3$-manifold, as discussed in the remark following the proof of \cite[Proposition 16]{lisitsa2023automated}. It is conjectured that a closed, connected $3$-manifold is an L-space if and only if its fundamental group has no left-orderable quotient, or equivalently by \cite[Proposition 10]{lisitsa2023automated} and \cite[Proposition 16]{lisitsa2023automated}, admits no left total preorder.

Additionally, we do not make the advance assumption that the boundary tori are incompressible. While these topological assumptions can be obtained at no cost through the prime and JSJ decompositions, we are not sure if they are the best way to describe the topological structure of L-spaces in the framework illustrated in Figure~\ref{fig:framework}. Therefore, we impose minimal topological properties to enhance generality.

We define order detection directly on the groups.

\begin{definition}
    Let $\le$ be a left total preorder on a group $G$. Let $\mu,\lambda$ be commuting group elements. We say $[x:y]$ on the real projective line $\mathbb{RP}^1$ is weakly $\le$-detected with respect to $(\mu,\lambda)$, if and only if $\mu^{p_2} \lambda^{q_2}\ge 1$ for all integers $p_1,q_1,p_2, q_2$ with $\mu^{p_1} \lambda^{q_1}\ge 1$ and $(p_1y - q_1 x)(p_2 y - q_2 x)>0$.
\end{definition}

Let $\varphi:\mathbb{Z}^2\to G$ be the homomorphism defined by $\varphi(p,q)=\mu^p\lambda^q$. If the image of $\varphi$ is in the residue group of $\le$, then every slope is weakly $\le$-detected. Otherwise, the pullback of the left total preorder $\le$ by $\varphi$ is a left total preorder on $\mathbb{Z}^2$. Just as the left orders on $\mathbb{Z}^2$ discussed in \cite[Subsection 2.2]{boyer2022order}, for every left total preorder on $\mathbb{Z}^2$, there exists a unique line in $\mathbb{R}^2$ such that the elements of $\mathbb{Z}^2$ that lie to one side of it are positive and the elements lying to the other side are negative. The slope of this line is the unique slope which is weakly $\le$-detected.

\begin{definition}
    Let $\le$ be a left total preorder on a group $G$. Let $\mu,\lambda$ be commuting group elements. We say $[p:q]$ ($p$ and $q$ are relatively prime integers) on the rational projective line $\mathbb{QP}^1$ is strongly $\le$-detected with respect to $(\mu,\lambda)$, if and only if every conjugate of $\mu^{p} \lambda^{q}$ is in the residue group of $\le$. 
\end{definition}

Because the strong order detection corresponds to Dehn filling in topology, we do not need to define the strong order detection for irrational slopes for our purpose. In fact, we have the following gluing theorem for strong order detection.

\begin{theorem}\label{thm:strong}
    Let $(\mu_i,\lambda_i)$ $(i\in I\cup J)$ be commuting pairs in a group $G$. Let $j_0$ be an index in $J$. Let $N$ denote the normal subgroup generated by $\mu_{j_0}^p \lambda_{j_0}^q$ ($p$ and $q$ are relatively prime integers). Let $s_i\in\mathbb{RP}^1$ $(i\in I)$ and $s_j\in\mathbb{QP}^1$ $(j\in J)$ be slopes with $s_{j_0}=[p:q]$. Then the following statements are equivalent:
    \begin{enumerate}[label=(\alph*)]
        \item there exists a left total preorder $\le_0$ on $G$ such that the slope $s_i$ is weakly $\le_0$-detected with respect to $(\mu_i,\lambda_i)$ for every $i\in I$, and the slope $s_j$ is strongly $\le_0$-detected with respect to $(\mu_j,\lambda_j)$ for every $j\in J$;
        \item there exists a left total preorder $\le$ on $G/N$ such that the slope $s_i$ is weakly $\le$-detected with respect to $(\mu_i N,\lambda_i N)$ for every $i\in I$, and the slope $s_j$ is strongly $\le$-detected with respect to $(\mu_j N,\lambda_j N)$ for every $j\in J\setminus\{j_0\}$.
    \end{enumerate}    
\end{theorem}
\begin{proof}
    Suppose that statement (b) holds. Then by definition, the pullback $\le_0$ of the left total preorder $\le$ on $G/N$ by the quotient map $G \to G/N$ is a left total preorder that strongly detects the slope $[p:q]$ with respect to $(\mu_{j_0},\lambda_{j_0})$. Because $x\le_0 y$ if and only if $xN\le y N$, the left total preorder $\le_0$ satisfies the condition in statement (a).

    Conversely, suppose that statement (a) holds. Then the normal subgroup $N$ is contained in the residue group of $\le_0$. Thus $\le_0$ naturally induces a left total preorder $\le$ on $G/N$. By definition, $x\le_0 y$ if and only if $xN\le y N$, so the left total preorder $\le$ satisfies the condition in statement (b).
\end{proof}

It follows from the definitions that, if a slope is strongly $\le$-detected, then it is weakly $\le$-detected with respect to every conjugate of the meridian-longitude pair. 

We also need to define a conjugacy invariant version of weak order detection. However, the regular order detection \cite{boyer2017foliations} which worked for Seifert fibered pieces may not be suitable for general $3$-manifolds with multiple torus boundary components. Instead, we consider weak order detection results with respect to every conjugate of a meridian-longitude pair. 

\begin{definition}\label{def:mix}
        Let $\le$ be a left total preorder on a group $G$. Let $(\mu_i,\lambda_i)$ $(i=0,1,\ldots,k)$ be commuting pairs. Let $D_i$ $(i=0,1,\ldots,k)$ be subsets of $\mathbb{RP}^1$. We say $(D_0;D_1,\ldots, D_k)$ is hybridly $\le$-detected with respect to $((\mu_0,\lambda_0);(\mu_1,\lambda_1), \ldots,(\mu_k,\lambda_k))$, if and only if
        \begin{enumerate}[label=(\alph*)]
            \item with respect to any conjugate of $(\mu_i,\lambda_i)$ $(i=1,\ldots,k)$, a slope in $D_i$ is weakly $\le$-detected, and,
        \item with respect to $(\mu_{0},\lambda_{0})$, a slope in $D_0$ is weakly $\le$-detected.
        \end{enumerate}
\end{definition}

The strength of condition (a) in the definition of hybrid order detection lies between that of weak order detection and regular order detection. When $D_i$ is a singleton $\{s_i\}$, this condition is equivalent to the slope $s_i$ being regularly $\le$-detected. When $D_i$ is $\mathbb{RP}^1\setminus\{s_i\}$, this condition is equivalent to the requirement that, for any conjugate order $\le'$ of $\le$, $s_i$ is not the only weakly $\le'$-detected slope.

The hybrid order detection corresponds to the topological operation of splicing two $3$-manifolds, each with multiple torus boundary components, along a torus boundary component. We provide the gluing theorem for hybrid order detection.

\begin{theorem}\label{thm:weak}
    Let $(\mu_{1,i},\lambda_{1,i})$ $(i=0,1,\ldots, k)$ commuting pairs in a group $G_1$, and $(\mu_{2,j},\lambda_{2,j})$ $(j=0,1,\ldots, l)$ be commuting pairs in a group $G_2$. Let $N$ denote the normal subgroup of $G_1* G_2$ generated by $\mu_{1,0}^{-1} \mu_{2,1}$ and $\lambda_{1,0}^{-1}\lambda_{2,1}$. Let $D_{1,i}$ $(i=0,1,\ldots, k)$ and $D_{2,j}$ $(j=0,1,\ldots, l)$ be subsets of $\mathbb{RP}^1$ such that $D_{1,0} \cup D_{2,1}=\mathbb{RP}^1$.
    
    Suppose that there exists a left total preorder $\le$ on $(G_1\ast G_2)/N$ such that \[(D_{2,0}; D_{1,1},\ldots, D_{1,k},D_{2,2},\ldots, D_{2,l})\]
    is hybridly $\le$-detected with respect to 
    \[((\mu_{2,0}N,\lambda_{2,0}N); (\mu_{1,1}N,\lambda_{1,1}N),\ldots,(\mu_{1,k}N,\lambda_{1,k}N),(\mu_{2,2}N,\lambda_{2,2}N),\ldots, (\mu_{2,l}N,\lambda_{2,l}N)).\]
    Then either 
    \begin{enumerate}[label=(\alph*)]
        \item there exists a left total preorder $\le_1$ on $G_1$ such that $(D_{1,0}; D_{1,1},\ldots,D_{1,k})$ is hybridly $\le_1$-detected with respect to $((\mu_{1,0},\lambda_{1,0});(\mu_{1,1},\lambda_{1,1}),\ldots,(\mu_{1,k},\lambda_{1,k}))$, or,
        \item there exists a left total preorder $\le_2$ on $G_2$ such that $(D_{2,0}; D_{2,1},\ldots,D_{2,l})$ is hybridly $\le_2$-detected with respect to $((\mu_{2,0},\lambda_{2,0});(\mu_{2,1},\lambda_{2,1}),\ldots,(\mu_{2,l},\lambda_{2,l}))$.
    \end{enumerate} 
\end{theorem}
\begin{proof}

    Let $\le_0$ denote the pullback of $\le$ by the quotient map $G_1\ast G_2 \to(G_1\ast G_2)/N$. Then $\le_0$ is a left total preorder on $G_1\ast G_2$ with $x\le_0 y$ if and only if $xN\le yN$ for any $x,y\in G_1\ast G_2$.

    Let $g$ be an arbitrary element in $G_2$. Let $\varphi_g: G_1\ast G_2 \to G_1\ast G_2$ denote the conjugation by $g$. Let $\le_g$ denote the pullback of $\le_0$ by the function composition of the inclusion map $G_1 \to G_1\ast G_2$ and the conjugation $\varphi_g: G_1\ast G_2 \to G_1\ast G_2$. Then either $\le_g$ is the trivial relation or $\le_g$ is a left total preorder on $G_1$. By definition, we have $x\le_g y$ if and only if $xg \le_0 yg$ for any $x,y\in G_1$. If condition (a) is not satisfied, then either $\le_g$ is the trivial relation or no slope in $D_{1,0}$ is weakly $\le_g$-detected with respect to $(\mu_{1,0},\lambda_{1,0})$. In either way, a slope in $D_{2,1}$ is weakly $\le_0$-detected with respect to $(g^{-1}\mu_{1,0}g,g^{-1}\lambda_{1,0}g)$. Because $N$ is in the residue group of $\le_0$, a slope in $D_{2,1}$ is weakly $\le_0$-detected with respect to $(g^{-1}\mu_{2,1}g,g^{-1}\lambda_{2,1}g)$.

    Let $\le_1$ and $\le_2$ denote the pullback of $\le_0$ by the inclusion maps $G_1\to G_1 \ast G_2$ and $G_2\to G_1 \ast G_2$ respectively. If $\le_2$ is the trivial relation, then by $x\le_2 y$ if and only if $x\le_0 y$ for any $x,y\in G_2$, the condition (b) is satisfied by $\le_2$. Now we assume that $\le_2$ is the trivial relation, then $G_2$ is contained in the residue group of $\le_0$. Thus $G_1$ is not contained in the residue group of $\le_0$, and $\le_1$ is a left total preorder with $x\le_1 y$ if and only if $x\le_0 y$ for any $x,y\in G_1$. Since $\mu_{1,0}$ and $\lambda_{1,0}$ are in the subgroup generated by $N$ and $G_2$, they are in the residue group of $\le_1$. Therefore condition (a) is satisfied by $\le_1$.
\end{proof}

We will prove the following order detection result for flat fully augmented chainmail links in Section~\ref{sec:4}. Then by Theorem~\ref{thm:strong}, \cite[Proposition 10]{lisitsa2023automated} and \cite[Proposition 16]{lisitsa2023automated}, the fundamental group of a $3$-manifold obtained from alternating surgeries on flat fully augmented chainmail links does not have left-orderable quotients. In particular, Theorem~\ref{main-1} holds. Furthermore, with Theorem~\ref{thm:strong} and Theorem~\ref{thm:weak}, we can prove the non-left-orderability of many other L-spaces constructed from splicing and Dehn fillings.

\begin{theorem}\label{thm:3.3-main}
    Let $G$ be a planar graph with vertex set $V$ and edge set $E$. Let $V_+$ be a subset of $V$ such that every connected component of $G$ contains at least one vertex in $V_+$. Let $v_o$ be a vertex of $V_+$. Let $L$ denote the flat fully augmented chainmail link with respect to $G$. There is no proper left total preorder $\le$ on $\pi_1(S^3\setminus L)$ such that:
    \begin{enumerate}[label=(\alph*)]
        \item for every edge $e\in E$ and every conjugate of a meridian-longitude pair of $\partial D(e)$, a negative slope is weakly $\le$-detected;
        \item for every vertex $v\in V\setminus V_+$ and a meridian-longitude pair of $\partial D(v)$, the slope $0$ is strongly $\le$-detected;
        \item for every vertex $v\in V_+\setminus\{v_o\}$ and every conjugate of a meridian-longitude pair of $\partial D(v)$, a positive slope is weakly $\le$-detected;
        \item a positive slope is weakly $\le$-detected with respect to a meridian-longitude pair of $\partial D(v_o)$.
    \end{enumerate}
\end{theorem}

\section{Flat fully augmented chainmail link}\label{sec:3}
Consider a finite planar graph $G\subset S^2$ with vertex set $V$ and edge set $E$. For each vertex $v\in V$, draw a planar disk $D(v) \subset S^2$, such that these disks are pairwise disjoint. For each edge $e=(v_1,v_2)\in E$, draw a simple planar curve $C(e)\subset S^2$ connecting a point in the interior of $D(v_1)$ and a point in the interior of $D(v_2)$, such that these curves are pairwise disjoint. Suppose that, if we shrink each disk to a vertex, then these vertices together with the curves $C(e)$ form the planar disk $G\subset S^2$. See Figure~\ref{fig:2} for an example.

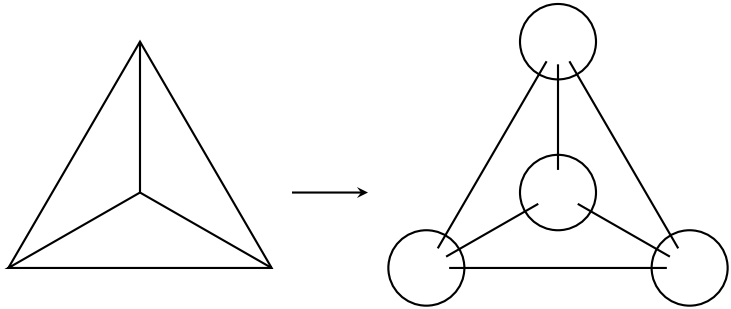
\begin{figure}[!htbp] 
\centering
\begin{tikzpicture}
\draw[thick] (-150:2) -- (90:2) -- (-30:2) -- (-150:2)
--(0:0); 
\draw[thick] (90:2)--(0:0) -- (-30:2);
\draw [-stealth,thick](2,0) -- (3,0);
\draw[thick] (5.5,0) circle [radius=0.5];
\draw[thick] (5.5,2) circle [radius=0.5];
\draw[thick] ({5.5-sqrt(3)},-1) circle [radius=0.5];
\draw[thick] ({5.5+sqrt(3)},-1) circle [radius=0.5];
\draw[thick] (5.5,0.3) -- (5.5,1.7);
\draw[thick] ({5.5-sqrt(3)*0.15},-0.15) -- 
                ({5.5-sqrt(3)*0.85},-0.85);
\draw[thick] ({5.5+sqrt(3)*0.15},-0.15) -- 
                ({5.5+sqrt(3)*0.85},-0.85);
\draw[thick] ({5.5+sqrt(3)-0.15},{-1+0.15*sqrt(3)}) -- 
            ({5.5+0.15},{2-0.15*sqrt(3)});
\draw[thick] ({5.5-sqrt(3)+0.15},{-1+0.15*sqrt(3)}) -- 
            ({5.5-0.15},{2-0.15*sqrt(3)});
\draw[thick] ({5.5-sqrt(3)+0.3},-1) -- ({5.5+sqrt(3)-0.3},-1);
\end{tikzpicture}
\caption{From a planar graph $G=(V,E)$ to disks $D(v)$ $(v\in V)$ and curves $C(e)$ $(e\in E)$.}\label{fig:2}
\end{figure}

Decompose the $3$-sphere $S^3$ as two $3$-balls $B_+$ and $B_-$ with common boundary $S^2$. For each edge $e\in E$, let $C_+(e)$ and $C_-(e)$ be two curves in $B_+$ and $B_-$ respectively sharing the same endpoints with $C(e)$, such that there exists an isotopoy of the curves $C_+(e)$ $(e\in E)$ (resp., the curves $C_-(e)$ $(e\in E)$) to the curves $C(e)$ $(e\in E)$ in $B_+$ (resp., in $B_-$). Then the link \[\bigcup_{v\in V} \partial D(v) \cup \bigcup_{e\in E} \left(C_+(e)\cup  C_-(e)\right)\] in $S^3$ is determined by the planar graph $G$. This link, denoted as $L$, is called a flat fully augmented chainmail link. See \cite[Figure 8]{agol2023chainmail} for an example. It is worth mentioning that Abchir and Sabak \cite{abchir2023infinite, abchir2023left} considered the orderability of the fundamental groups of L-space surgeries on an augmented alternating link. The main difference is that they only considered the augmentation of an alternating link with one circle, therefore it is not fully augmented.

The knot components $\partial D(v)$ $(v\in V)$ are called vertex loops, while $C_+(e)\cup  C_-(e)$ $(e\in E)$ are called edge loops or crossing loops. For convenience, for each $e\in E$, let $D(e)$ denote a disk with boundary $C_+(e)\cup  C_-(e)$ and intersecting $S^2\subset S^3$ at $C(e)$. We can assume that the disks $D(e)$ $(e\in E)$ are pairwise disjoint. 

A presentation of the link group of $L$ can be computed using the Wirtinger presentation. We follow the same strategy as in the proof of \cite[Theorem]{agol2023chainmail} to derive a group presentation and to identify the meridian and longitude cycles. 

We orient the edges in $E$ arbitrarily, and then specify an orientation of $L$ as follows.
\begin{enumerate}[label=(\alph*)]
    \item Each vertex loop is oriented counterclockwise;
    \item For each edge $e$ going from $v_1$ to $v_2$, orient the edge loop $\partial D(e)$ such that it is positively linked with $\partial D(v_1)$. 
\end{enumerate}
Then we can check that the following group elements generate the link group.
\begin{enumerate}[label=(\alph*)]
    \item For each edge $e$, let $\mu_e\in \pi_1(S^3\setminus L)$ be represented by a loop in the interior of $B_+$ intersecting $D(e)$ once positively, and disjoint from any $D(e')$ with $e'\in E\setminus \{e\}$.  
    \item For each edge $e$ going from $v_1$ to $v_2$, let $\lambda_e \in \pi_1(S^3\setminus L)$ be represented by a loop intersecting $D(v_1)$ once positively, intersecting $D(v_2)$ once negatively, and having no other intersections with $S^2\subset S^3$ or any $D(e')$ $(e'\in E)$.
    \item For a vertex $v$ incident with an edge $e$, let $\mu_{v,e} \in \pi_1(S^3\setminus L)$ be represented by a loop intersecting $D(v)$ once positively, intersecting $S^2\setminus D(v)$ on the clockwise side of $C(e)$ once negatively, and having no other intersections with $S^2\subset S^3$ or any $D(e')$ $(e'\in E)$.
\end{enumerate}
Furthermore, the elements $\mu_e$ and $\lambda_e$ are represented by the meridian and longitude of the edge loop $\partial D(e)$.

For each vertex $v$, let $e_1, e_2,\ldots, e_k$ be edges adjacent to $v$ in counterclockwise order. Define $\varepsilon_1,\varepsilon_2\ldots, \varepsilon_k\in \{-1,1\}$ by
    \[\varepsilon_i=\begin{cases}
        -1 &\mbox{, if } v \mbox{ is the head of }e_i,\\ 
        1 &\mbox{, if } v \mbox{ is the tail of }e_i.\\ 
    \end{cases}\]
Then the elements $\mu_{v,e_1}$ and $\mu_{e_1}^{\varepsilon_1}\mu_{e_2}^{\varepsilon_2}\cdots \mu_{e_k}^{\varepsilon_k}$ are represented by the meridian and longitude of the vertex loop $\partial D(v)$. By the relations \[\mu_{v,e_i}=\mu_{e_i}^{\varepsilon_i}\mu_{v,e_{i+1}}\mu_{e_i}^{-\varepsilon_i},\] the pairs $(\mu_{v,e_i},\mu_{e_i}^{\varepsilon_i}\mu_{e_{i+1}}^{\varepsilon_{i+1}}\cdots \mu_{e_{i-1}}^{\varepsilon_{i-1}})$ $(i=1,2,\ldots, k)$ are conjugate with each other.

For each edge $e$ going from $v_1$ to $v_2$, the relation \[ \mu_{v_2,e}\lambda_e =\mu_e^{-1} \mu_{v_1,e} \mu_e\] holds. Equivalently, we have \[\mu_{v_1,e}\lambda_e^{-1} =\mu_e \mu_{v_2,e} \mu_e^{-1}.\]

Let $v_1, e_1, v_2, e_2, \ldots, v_k, e_k$ be a sequence denoting an undirected cycle in $G$. Define $\varepsilon_1,\varepsilon_2\ldots, \varepsilon_k\in \{-1,1\}$ by
    \[\varepsilon_i=\begin{cases}
        -1 &\mbox{, if } e_i \mbox{ goes from  }v_i \mbox{ to } v_{i+1},\\ 
        1 &\mbox{, if } e_i \mbox{ goes from }v_{i+1}\mbox{ to } v_{i}.\\ 
    \end{cases}\]
Then the cycle relation \[\lambda_{e_1}^{\varepsilon_1}\lambda_{e_2}^{\varepsilon_2}\cdots\lambda_{e_k}^{\varepsilon_k}=1\] holds.

\section{Proof of Theorem~\ref{thm:3.3-main}}\label{sec:4}
    Let $\le$ be a proper left total preorder on $\pi_1(S^3\setminus L)$ satisfying all the conditions. For every $g\in \pi_1(S^3\setminus L)$, we define the conjugate of $\le$ by $g$, denoted by $\le_g$, as $x\le_g y$ if and only if $g x\le g y$. Then $\le_g$ is a proper left total preorder on $\pi_1(S^3\setminus L)$ satisfying conditions (a), (b) and (c), and also condition (d) with respect to a conjugate of the meridian-longitude pair. Hence, we may specify the meridian-longitude pair of $\partial D(v_o)$ in condition (d) to be the pair defined in Section~\ref{sec:3}  without loss of generality.
    
    Consider a direct graph $\tilde{G}=(V,\tilde{E})$ with an edge \[\tilde{e}:=\begin{cases}
        (v_1, v_2)&\mbox{, if } \lambda_e\le 1\mbox{ does not hold;}\\
        (v_2, v_1)&\mbox{, if } 1\le\lambda_e\mbox{ does not hold.}
    \end{cases}\] for each $e=(v_1,v_2)\in E$ such that $\lambda_e$ is not in the residue group of $\le$. We prove that $\tilde{G}$ is a directed acyclic graph using the cycle relations by contradiction. 

    \begin{lemma}\label{lem:DAG}
        The graph $\tilde{G}$ does not have directed cycles.
    \end{lemma}
    \begin{proof}
       Suppose that $v_1,\tilde{e}_1, v_2,\tilde{e}_2,\ldots, v_k, \tilde{e}_k$ is a sequence denoting a directed cycle in $\tilde{G}$, and $\varepsilon_1,\varepsilon_2\ldots, \varepsilon_k$ are exponents in $\{-1,1\}$ such that
    \[\varepsilon_i=\begin{cases}
        -1 &\mbox{, if } e_i \mbox{ goes from  }v_i \mbox{ to } v_{i+1},\\ 
        1 &\mbox{, if } e_i \mbox{ goes from }v_{i+1}\mbox{ to } v_{i}.\\ 
    \end{cases}\]
    Then for each $i=1,2\ldots, k$, the inequality $1\le \lambda_{e_i}^{\varepsilon_i}$ does not hold. In particular, we have $\lambda_{e_i}^{\varepsilon_i}\le 1$ for $i=2,3,\ldots, k-1$. By the cycle relation, we have $1=\lambda_{e_1}^{\varepsilon_1}\lambda_{e_2}^{\varepsilon_2}\ldots \lambda_{e_k}^{\varepsilon_k}\le \lambda_{e_1}^{\varepsilon_1}$, which leads to a contradiction. 
    \end{proof}
    
    For each $v\in V$, let $e_1, e_2,\ldots, e_k$ be edges adjacent to $v$ in counterclockwise order. Define $\varepsilon_1,\varepsilon_2\ldots, \varepsilon_k\in \{-1,1\}$ by
    \[\varepsilon_i=\begin{cases}
        -1 &\mbox{, if } v \mbox{ is the head of }e_i,\\ 
        1 &\mbox{, if } v \mbox{ is the tail of }e_i.\\ 
    \end{cases}\]
    We consider the following statements for each vertex $v$.
    \begin{enumerate}
        \item[\textbf{S1.}] There exists a nonnegative integer $n^{+}_i$ such that $\mu_{e_i}^{\varepsilon_i}\ge \mu_{v,e_i}^{n^+_i}$.
        \item[\textbf{S2.}] There exists a negative integer $n^{-}_i$ such that $\mu_{e_i}^{\varepsilon_i}\le  \mu_{v,e_i}^{n^-_i}$.
        \item[\textbf{S3.}] The inequality $\mu_{v,e_i}\le 1$ holds.
    \end{enumerate}
    We prove two lemmas about \textbf{S1}, \textbf{S2}, and \textbf{S3}.
    
    \begin{lemma}\label{lem:2,3->1}
        Suppose that \textbf{S2} and \textbf{S3} hold for all vertices $v'$ such that there exists an edge going from $v'$ to $v$ in $\tilde{G}$. Then \textbf{S1} holds for $v$.
    \end{lemma}
    \begin{proof}
        Suppose that $e_i$ connects $v$ to another vertex $v'$. We consider two scenarios.

        First, suppose that $\lambda_{e_i}^{\varepsilon_i}\ge 1$. By condition (a), a negative slope is weakly $\le$-detected with respect to the meridian-longitude pair $(\lambda_{e_i},\mu_{e_i})$. Thus we have $\mu_{e_i}^{\varepsilon_i}\ge 1$, and \textbf{S1} holds with $n_i^{+}=0$.

        Second, suppose that $\lambda_{e_i}^{\varepsilon_i}\ge 1$ does not hold. By definition, the edge $\tilde{e}_i$ goes from $v'$ to $v$. By \textbf{S2} for $v'$, there exists a negative integer $n'$ such that
        \[\mu_{e_i}^{-\varepsilon_i}\le (\mu_{e_i}^{-\varepsilon_i} \mu_{v,e_i} \mu_{e_i}^{\varepsilon_i}\lambda_{e_i} ^{-\varepsilon_i})^{n'},\]
        which is equivalent to 
        \[(\lambda_{e_i}^{\varepsilon_i} \mu_{v,e_i}^{-1} )^{-n'} \mu_{e_i}^{\varepsilon_i}\ge 1.\]
        If $\mu_{e_i}^{\varepsilon_i}\ge \mu_{v,e_i}^{-n'}$, then \textbf{S1} holds for $n_i^+ =-n'$. Otherwise, we have 
        \begin{align*}
            &\lambda_{e_i}^{\varepsilon_i}(\mu_{v,e_i}^{-1}\lambda_{e_i}^{\varepsilon_i}\mu_{v,e_i})(\mu_{v,e_i}^{-2}\lambda_{e_i}^{\varepsilon_i}\mu_{v,e_i}^2)\cdots(\mu_{v,e_i}^{n'+1}\lambda_{e_i}^{\varepsilon_i}\mu_{v,e_i}^{-n'-1})\\
            =& (\lambda_{e_i}^{\varepsilon_i} \mu_{v,e_i}^{-1} )^{-n'} \mu_{v,e_i}^{-n'}\\
            \ge& (\lambda_{e_i}^{\varepsilon} \mu_{v,e_i}^{-1} )^{-n'} \mu_{e_i}^{\varepsilon_i}\\
            \ge &1.
        \end{align*}
        In this case, there exists a nonnegative integer $n''\in\{0,1,2,\ldots -n'-1\}$ such that \[\mu_{v,e_i}^{-n''}\lambda_{e_i}^{\varepsilon_i}\mu_{v,e_i}^{n''}\ge 1.\]
        By condition (a), a negative slope is weakly $\le$-detected with respect to the meridian-longitude pair $(\mu_{v,e_i}^{-n''}\mu_{e_i}\mu_{v,e_i}^{n''},\mu_{v,e_i}^{-n''}\lambda_{e_i}\mu_{v,e_i}^{n''})$. Thus we have \[\mu_{v,e_i}^{-n''}\mu_{e_i}^{\varepsilon_i}\mu_{v,e_i}^{n''}\ge 1.\]
        By \textbf{S3} for $v'$, we have \[\mu_{v,e_i} \lambda_{e_i} ^{-\varepsilon_i}\le 1.\]
        Thus we have 
        \[\mu_{e_i}^{\varepsilon_i}=\mu_{v,e_i}^{n''}(\mu_{v,e_i}^{-n''}\mu_{e_i}^{\varepsilon_i}\mu_{v,e_i}^{n''})(\lambda_{e_i}^{-\varepsilon_i}(\lambda_{e_i}^{\varepsilon_i}\mu_{v,e_i}^{-1}))^{n''}\ge \mu_{v,e_i}^{n''}.\]
        In this case \textbf{S1} holds for $n^+_i=n''$. 
    \end{proof}

    \begin{lemma}\label{lem:1->2,3}
        Suppose that \textbf{S1} holds for $v$. Then either $v$ is in $V\setminus V_+$ and isolated in $\tilde{G}$, or \textbf{S2} and \textbf{S3} hold for $v$. 
    \end{lemma}
    \begin{proof}
        Because $\mu_{e_i}^{\varepsilon_i}\ge \mu_{v,e_i}^{n^+_i}$ holds for $i=1,2,\ldots, k$, we have 
        \[\mu_{v,e_1}^{-\sum_{j=1}^k n_j^+}(\mu_{e_1}^{\varepsilon_1} \mu_{e_{2}}^{\varepsilon_{2}} \cdots \mu_{e_{k}}^{\varepsilon_{k}})=(\mu_{v,e_1}^{-n^+_1}\mu_{e_1}^{\varepsilon_1})(\mu_{v,e_{2}}^{-n^+_{2}}\mu_{e_{2}}^{\varepsilon_{2}}) \cdots (\mu_{v,e_{k}}^{-n^+_{k}}\mu_{e_{k}}^{\varepsilon_{k}})\ge 1. \]
        With respect to the meridian-longitude pair $(\mu_{v,e_1}, \mu_{e_1}^{\varepsilon_1} \mu_{e_{2}}^{\varepsilon_{2}} \cdots \mu_{e_{k}}^{\varepsilon_{k}})$, either $v\in V\setminus V_+$ and $0$ is strongly $\le$-detected by condition (b), or a positive slope is weakly $\le$-detected by conditions (c) and (d).

        First, suppose that $v\in V\setminus V_+$ and $0$ is strongly $\le$-detected, and $n_i^+=0$ for each $i=1,2,\ldots,k$. Then we have $\mu_{e_i}^{\varepsilon_i} \ge 1$ for $i=1,2,\ldots, k$ and $\mu_{e_1}^{\varepsilon_1} \mu_{e_{2}}^{\varepsilon_{2}} \cdots \mu_{e_{k}}^{\varepsilon_{k}}\le 1$. In this case, we have \[1\le \mu_{e_{i+1}}^{\varepsilon_{i+1}}\ldots\mu_{e_{k}}^{\varepsilon_{k}}(\mu_{e_1}^{\varepsilon_1} \mu_{e_{2}}^{\varepsilon_{2}} \cdots \mu_{e_{k}}^{\varepsilon_{k}})^{-1}\mu_{e_{1}}^{\varepsilon_{1}}\ldots\mu_{e_{i-1}}^{\varepsilon_{i-1}}=\mu_{e_i}^{-\varepsilon_i} \le 1,\]
        which implies $1\le \mu_{e_i}\le 1$. By condition (a), a negative slope is weakly $\le$-detected with respect to the meridian-longitude pair $(\mu_{e_i},\lambda_{e_i})$. Thus we have $1\le \lambda_{e_i}\le 1$. In this case, the vertex $v$ is isolated in $\tilde{G}$.

        Second, suppose that either a positive slope is weakly $\le$-detected, or at least one $n_i^+$ is nonzero. Then there exists a negative integer $n'$ such that \[(\mu_{e_1}^{\varepsilon_1} \mu_{e_{2}}^{\varepsilon_{2}} \cdots \mu_{e_{k}}^{\varepsilon_{k}})^{-1}\mu_{v,e_1}^{n'}\ge 1.\]
        Thus we have \begin{align*}
             & \mu_{v,e_i}^{n'+n_i^+-\sum_{j=1}^k n_j^+}\\
             =& \mu_{e_i}^{\varepsilon_i}(\mu_{v,e_i}^{-n^+_i}\mu_{e_i}^{\varepsilon_i})^{-1}\mu_{v,e_i}^{n'-\sum_{j=1}^k n_j^+}\\
             =& \mu_{e_i}^{\varepsilon_i}(\mu_{v,e_i}^{-n^+_i}\mu_{e_i}^{\varepsilon_i})^{-1}\cdots (\mu_{v,e_1}^{-n^+_1}\mu_{e_1}^{\varepsilon_1})^{-1} \mu_{v,e_1}^{n'-\sum_{j=1}^k n_j^+}(\mu_{v,e_{1}}^{-n^+_{1}}\mu_{e_{1}}^{\varepsilon_{1}}) \cdots (\mu_{v,e_{i-1}}^{-n^+_{i-1}}\mu_{e_{i-1}}^{\varepsilon_{i-1}})\\
            =&\mu_{e_i}^{\varepsilon_i}(\mu_{v,e_{i+1}}^{-n^+_{i+1}}\mu_{e_{i+1}}^{\varepsilon_{i+1}}) \cdots (\mu_{v,e_{k}}^{-n^+_{k}}\mu_{e_{k}}^{\varepsilon_{k}})(\mu_{e_1}^{\varepsilon_1} \mu_{e_{2}}^{\varepsilon_{2}} \cdots \mu_{e_{k}}^{\varepsilon_{k}})^{-1}\mu_{v,e_1}^{n'}(\mu_{v,e_{1}}^{-n^+_{1}}\mu_{e_{1}}^{\varepsilon_{1}}) \cdots (\mu_{v,e_{i-1}}^{-n^+_{i-1}}\mu_{e_{i-1}}^{\varepsilon_{i-1}})\\
            \ge& \mu_{e_i}^{\varepsilon_i}.
        \end{align*}
        Then \textbf{S2} holds with $n_i^- =n'+n_i^+-\sum_{j=1}^k n_j^+$. By $\mu_{v,e_i}^{n_i^+}\le \mu_{e_i}^{\varepsilon_i}\le \mu_{v,e_i}^{n_i^-}$, we have $\mu_{v,e_i}\le 1$. Therefore \textbf{S3} also holds in this case. 
    \end{proof}

    Because $\tilde{G}$ is a directed acyclic graph by Lemma~\ref{lem:DAG}, we can list the vertices according to a topological ordering of $\tilde{G}$. By induction over the topological ordering of $\tilde{G}$, we can prove that \textbf{S1}, \textbf{S2} and \textbf{S3} hold for every non-isolated vertex in $\tilde{G}$ and every vertex in $V_+$ by Lemma~\ref{lem:2,3->1} and Lemma~\ref{lem:1->2,3}.

    Now we reverse the inequality signs. While every edge in $\tilde{G}$ is reversed, the set of isolated vertices remains invariant. Since a slope is weakly (resp., strongly) $\le$-detected if and only if it is weakly (resp., strongly) $\ge$-detected, we can check that $\ge$ satisfies the conditions in Theorem~\ref{thm:3.3-main} with respect to the same meridian-longitude pair of $\partial D(v_o)$. Therefore \textbf{S1}, \textbf{S2} and \textbf{S3} with reversed inequality signs also hold for every non-isolated vertex in $\tilde{G}$ and every vertex in $V_+$. 
    
    It follows that, for a non-isolated vertex $v$ in $\tilde{G}$ and a vertex $v$ in $V_+$, and any edge $e$ adjacent to $v$, the elements $\mu_{v,e}$ and $\mu_{e}$ are in the residue group of $\le$. By condition (a), a negative slope is weakly $\le$-detected with respect to the meridian-longitude pair $(\mu_e, \lambda_e)$. Thus $\lambda_e$ is also in the residue group of $\le$. 
    
    Let $V_1$ denote the set of vertices $v\in V$ such that the elements $\mu_{v,e}$, $\mu_{e}$, and $\lambda_e$ are in the residue group of $\le$ for every edge $e$ adjacent to $v$, then $V_+\subseteq V_1$. We prove that there does not exist an edge $e$ in $G$ connecting a vertex $v_1\in V_1$ and a vertex $v_2\in V\setminus V_1$ by contradiction.

    Suppose that there exists an edge $e$ in $G$ connecting $v_1\in V_1$ and $v_2\in V\setminus V_1$. By $v_1 \in V_1$, the elements $\mu_{v_1,e}$ $\mu_e$ and $\lambda_e$ are in the residue group of $\le$. By the group relation
    \[ \mu_{v_2,e}\lambda_e =\mu_e^{-1} \mu_{v_1,e} \mu_e\] or \[\mu_{v_1,e}\lambda_e^{-1} =\mu_e \mu_{v_2,e} \mu_e^{-1},\]
    the element $\mu_{v_2,e}$ is also in the residue group of $\le$. By $v_2\in V\setminus V_1$, the vertex $v_2$ is isolated in $\tilde{G}$. So $\lambda_{e_i}$ is in the residue group of $\le$ for every edge $e_i$ $(i=1,2,\ldots,k)$ adjacent to $v_2$. By condition (a), a negative slope is weakly $\le$-detected with respect to the meridian-longitude pair $(\mu_{e_i}, \lambda_{e_i})$. Thus $\mu_{e_i}$ is also in the residue group of $\le$. By the group relations (see Section~\ref{sec:3})  \[\mu_{v_2,e_i}=\mu_{e_i}^{\varepsilon_i}\mu_{v_2,e_{i+1}}\mu_{e_i}^{-\varepsilon_i}\]
    for $i=1,2,\ldots,k$, every $\mu_{v_2,e_i}$ is also in the residue group of $\le$. Then by definition, we have $v_2\in V_1$, which leads to a contradiction.

    Because every connected component of $G$ contains at least one vertex in $V_+$ and $V_+\subseteq V_1$, we have $V_1 = V$. Because the elements $\mu_{v,e}$, $\mu_{e}$, and $\lambda_e$ generate the group $\pi_1(S^3\setminus L)$, the left total preorder $\le$ cannot be proper. Therefore Theorem~\ref{thm:3.3-main} holds true.
\bibliographystyle{alpha}
\bibliography{ref}

\begin{thebibliography}{HRRW20}

\bibitem[Ago23]{agol2023chainmail}
Ian Agol.
\newblock Chainmail links and {L}-spaces.
\newblock {\em arXiv preprint arXiv:2306.10918}, 2023.

\bibitem[AS23a]{abchir2023infinite}
Hamid Abchir and Mohammed Sabak.
\newblock Infinite families of non-left-orderable {$L$}-spaces.
\newblock {\em Osaka Journal of Mathematics}, 60(1):77--103, 2023.

\bibitem[AS23b]{abchir2023left}
Hamid Abchir and Mohammed Sabak.
\newblock On left-orderability of involutory quandles of links.
\newblock {\em arXiv preprint arXiv:2310.05735}, 2023.

\bibitem[Ba19]{ba2019spaces}
Idrissa Ba.
\newblock L-spaces, left-orderability and two-bridge knots.
\newblock {\em Journal of Knot Theory and Its Ramifications}, 28(03):1950019,
  2019.

\bibitem[BC17]{boyer2017foliations}
Steven Boyer and Adam Clay.
\newblock Foliations, orders, representations, {L-spaces} and graph manifolds.
\newblock {\em Advances in Mathematics}, 310:159--234, 2017.

\bibitem[BC22]{boyer2022order}
Steven Boyer and Adam Clay.
\newblock Order-detection of slopes on the boundaries of knot manifolds.
\newblock {\em arXiv preprint arXiv:2206.00848}, 2022.

\bibitem[BGH21]{boyer2021slope}
Steven Boyer, Cameron~McA Gordon, and Ying Hu.
\newblock Slope detection and toroidal $3$-manifolds.
\newblock {\em arXiv preprint arXiv:2106.14378}, 2021.

\bibitem[BGW13]{boyer2013spaces}
Steven Boyer, Cameron~McA Gordon, and Liam Watson.
\newblock On {L}-spaces and left-orderable fundamental groups.
\newblock {\em Mathematische Annalen}, 356(4):1213--1245, 2013.

\bibitem[DPT05]{dkabkowski2005non}
Mieczys{\l}aw~K D{\k{a}}bkowski, J{\'o}zef~H Przytycki, and Amir~A Togha.
\newblock Non-left-orderable $3$-manifold groups.
\newblock {\em Canadian Mathematical Bulletin}, 48(1):32--40, 2005.

\bibitem[Eft18]{eftekhary2018bordered}
Eaman Eftekhary.
\newblock Bordered {Floer} homology and existence of incompressible tori in
  homology spheres.
\newblock {\em Compositio Mathematica}, 154(6):1222--1268, 2018.

\bibitem[GL16]{greene2016strong}
Joshua Greene and Adam Levine.
\newblock Strong {Heegaard} diagrams and strong {L}--spaces.
\newblock {\em Algebraic \& Geometric Topology}, 16(6):3167--3208, 2016.

\bibitem[Gre13]{greene2013spanning}
Joshua~Evan Greene.
\newblock A spanning tree model for the {Heegaard Floer} homology of a branched
  double-cover.
\newblock {\em Journal of Topology}, 6(2):525--567, 2013.

\bibitem[Gre18]{greene2018alternating}
Joshua Greene.
\newblock Alternating links and left-orderability.
\newblock {\em Proceedings of the American Mathematical Society},
  146(6):2707--2709, 2018.

\bibitem[HL16]{hedden2016splicing}
Matthew Hedden and Adam~Simon Levine.
\newblock Splicing knot complements and bordered {Floer} homology.
\newblock {\em Journal f{\"u}r die reine und angewandte Mathematik (Crelles
  Journal)}, 2016(720):129--154, 2016.

\bibitem[HRRW20]{hanselman2020spaces}
Jonathan Hanselman, Jacob Rasmussen, Sarah~Dean Rasmussen, and Liam Watson.
\newblock L-spaces, taut foliations, and graph manifolds.
\newblock {\em Compositio Mathematica}, 156(3):604--612, 2020.

\bibitem[HRW16]{hanselman2016bordered}
Jonathan Hanselman, Jacob Rasmussen, and Liam Watson.
\newblock Bordered floer homology for manifolds with torus boundary via
  immersed curves.
\newblock {\em arXiv preprint arXiv:1604.03466}, 2016.

\bibitem[Ito13]{ito2013non}
Tetsuya Ito.
\newblock Non-left-orderable double branched coverings.
\newblock {\em Algebraic \& Geometric Topology}, 13(4):1937--1965, 2013.

\bibitem[Li22]{li2022taut}
Tao Li.
\newblock Taut foliations of $3$-manifolds with {Heegaard} genus two.
\newblock {\em arXiv preprint arXiv:2202.00737}, 2022.

\bibitem[LL13]{levine2013strong}
Adam~Simon Levine and Sam Lewallen.
\newblock Strong {L}-spaces and left-orderability.
\newblock {\em Mathematical Research Letters}, 19(6):1237--1244, 2013.

\bibitem[LN22]{lin2022non}
Lin Li and Zipei Nie.
\newblock Non-left-orderability of cyclic branched covers of pretzel knots
  $p(3,-3,-2k-1)$.
\newblock {\em Proceedings of the Japan Academy, Series A: Mathematical
  Sciences}, 98(10), 2022.

\bibitem[LNV23]{lisitsa2023automated}
Alexei Lisitsa, Zipei Nie, and Alexei Vernitski.
\newblock Automated reasoning for proving non-orderability of groups.
\newblock {\em arXiv preprint arXiv:2310.05891}, 2023.

\bibitem[MP94]{matveev1994geometrical}
S~Matveev and M~Polyak.
\newblock A geometrical presentation of the surface mapping class group and
  surgery.
\newblock {\em Communications in mathematical physics}, 160:537--550, 1994.

\bibitem[Nie21]{nie2021explicit}
Zipei Nie.
\newblock An explicit description of $(1, 1)$ {L}-space knots, and
  non-left-orderable surgeries.
\newblock {\em Communications of Huawei Research}, 1:212--219, 2021.

\bibitem[OS04a]{ozsvath2004heegaard}
Peter Ozsv{\'a}th and Zolt{\'a}n Szab{\'o}.
\newblock Heegaard diagrams and holomorphic disks.
\newblock In {\em Different faces of geometry}, pages 301--348. Springer, 2004.

\bibitem[OS04b]{ozsvath2004lectures}
Peter Ozsv{\'a}th and Zolt{\'a}n Szab{\'o}.
\newblock Lectures on {Heegaard Floer homology}.
\newblock {\em Floer homology, gauge theory, and low-dimensional topology},
  5:29--70, 2004.

\bibitem[Pol14]{polyak20143}
M~Polyak.
\newblock From $3$-manifolds to planar graphs and cycle-rooted trees, talk at
  {Arnold}'s legacy conference.
\newblock {\em Fields Institute, Toronto}, 2014.

\bibitem[Ras17]{rasmussen2017space}
Sarah~Dean Rasmussen.
\newblock L-space intervals for graph manifolds and cables.
\newblock {\em Compositio Mathematica}, 153(5):1008--1049, 2017.

\end{thebibliography}
\end{document}